\documentclass[sn-mathphys,Numbered]{sn-jnl}

\usepackage{graphicx}
\usepackage{amsmath}
\usepackage{amssymb}
\usepackage{amsfonts}
\usepackage{mathrsfs}
\usepackage{mathtools}
\usepackage{setspace}

\theoremstyle{thmstyleone}
\newtheorem{theorem}{Theorem}[section]
\newtheorem{proposition}[theorem]{Proposition}
\newtheorem{lemma}[theorem]{Lemma}
\newtheorem{corollary}[theorem]{Corollary}

\theoremstyle{thmstylethree}

\theoremstyle{thmstyletwo}
\newtheorem{remark}{Remark}[section]

\makeatletter
\@addtoreset{equation}{section}
\makeatother
\renewcommand{\theequation}{\thesection.\arabic{equation}}
\newcommand{\subclass}[1]{\par\addvspace\smallskipamount\noindent\textbf{Mathematics Subject Classification (2020):}\enspace#1}
\makeatletter
\def\email#1{\global\advance\emailcnt by 1\relax%
\if@corauemail%
   \g@addto@macro\corrauthemail{%
   \setcounter{footnote}{0}%
   \textcolor{blue}{#1}%
   }%
\else%
   \g@addto@macro\authemail{%
   \setcounter{footnote}{0}%
   \textcolor{blue}{#1}%
   }%
\fi}
\makeatother

\author{\fnm{Pablo} \sur{Fern\'andez Refolio}}
\email{pablezfr@gmail.com}

\affil{\orgaddress{\city{Madrid}, \country{Spain}}}

\begin{document}

\title{Analytical Evaluation of Signature 2 Ramanujan-Type Series via Elementary Transformations and Companion Harmonic Expansions}

\abstract{We provide an explicit analytical evaluation of the known rational Ramanujan-type series for $1/\pi$ of signature $2$. Focusing on the singular moduli $k_r$ for $r \in \{2, 3, 4, 7\}$, we show that these invariants and their underlying elliptic identities follow from elementary transformations of complete elliptic integrals of the first and second kind. Using this framework, we derive two companion families of rational harmonic Ramanujan-type series for the reciprocal of $\pi$ involving higher-degree polynomials.}


\maketitle

\subclass{33C05, 33E05, 11F03.} 

\section{Introduction}
The study of Ramanujan-type series for $1/\pi$ within the framework of hypergeometric signatures is well established, with signature $2$ playing a central role via the classical theory of elliptic integrals and singular moduli (see, e.g., \cite{baruahberndt2009eisenstein} for a comprehensive historical overview and background). 

This paper provides an explicit analytical evaluation of the known rational Ramanujan-type series for $1/\pi$ of signature $2$. Focusing on the singular moduli $k_{r}$ for $r \in \{2, 3, 4, 7\}$, we show that the construction of these invariants and their underlying elliptic identities follow from well known elementary transformations (see, e.g., \cite{almkvistberndt1988}). Finally, we use these results to derive two companion families of rational harmonic Ramanujan-type series involving higher-degree polynomials.

\section{Preliminaries and Results}
The Euler Gamma function, denoted by $\Gamma(z)$, is defined via the following convergent improper integral for $\Re(z) > 0$
\begin{equation*}\label{eq:gamma_integral}
    \Gamma(z) = \int_0^\infty t^{z-1} e^{-t} \, dt.
\end{equation*}
To frame the complete elliptic integrals within the theory of hypergeometric functions, we recall Gauss's hypergeometric series~${}_2F_1(a,b;c;x)$ \cite{gauss1813}, which is defined for $0 < |x| < 1$ as
\begin{equation*}\label{eq:gauss_hypergeometric}
    {}_2F_1(a,b;c;x) = \sum_{m=0}^{\infty} \frac{(a)_m (b)_m}{(c)_m} \frac{x^m}{m!},
\end{equation*}
where $(a)_m = \Gamma(a+m)/\Gamma(a)$ denotes the rising factorial or Pochhammer symbol.
The complete elliptic integrals of the first and second kind, denoted by $K(k)$ and $E(k)$ respectively, admit the following equivalent representations via power series and hypergeometric expansions
\begin{equation}\label{eq:K_expansion}
    K(k)  \coloneqq  \int_{0}^{\pi/2} \frac{dt}{\sqrt{1-k^2\sin^2t}} = \frac{\pi}{2}\sum_{m=0}^{\infty}\frac{(2m)!^2}{2^{4m}m!^4}k^{2m} = \frac{\pi}{2} \, {_2F_1}\left(\frac{1}{2}, \frac{1}{2}; 1; k^2\right),
\end{equation}
with $k \in (0, 1)$ and
\begin{equation}\label{eq:E_expansion}
    E(k)  \coloneqq  \int_{0}^{\pi/2} \sqrt{1-k^2\sin^2t} \, dt = \frac{\pi}{2}\sum_{m=0}^{\infty}\frac{(2m)!^2}{2^{4m}(1-2m)m!^4}k^{2m} = \frac{\pi}{2} \, {_2F_1}\left(-\frac{1}{2}, \frac{1}{2}; 1; k^2\right), 
\end{equation}
with $k \in (0, 1]$.
These integrals satisfy Legendre's relation \cite{legendre1825} 
\begin{equation}\label{eq:legendre_relation}
    K(k)E(k') + K(k')E(k) - K(k)K(k') = \frac{\pi}{2}, \quad k \in (0,1),
\end{equation}
where $k' = \sqrt{1-k^2}$ represents the complementary modulus.
The standard derivatives of $K(k)$ and $E(k)$ are given by
\begin{equation}\label{eq:derke}
  \frac{\mathrm{d}K}{\mathrm{d}k}=\frac{E(k)}{k(1-k^2)} - \frac{K(k)}{k},\hspace{.5cm}\frac{\mathrm{d}E}{\mathrm{d}k}=\frac{E(k)-K(k)}{k}.
\end{equation}
These complete elliptic integrals satisfy Landen descending transformations \cite{almkvistberndt1988, landen1775}
\begin{equation}\label{eq:landendownwardke}
    K(\frac{2\sqrt{k}}{1+k})=K(k)(1+k),\hspace{.5cm}E(\frac{2\sqrt{k}}{1+k})=\frac{2E(k)}{1+k}+K(k)(k-1),
\end{equation}
the Landen ascending transformations
\begin{equation}\label{eq:landeupwardke}
   K\left(\frac{1-k'}{1+k'}\right)= \frac{1+k'}{2}K(k),\hspace{.5cm}E(\frac{1-k'}{1+k'})=\frac{E(k)}{1+k'}+\frac{k'K(k)}{1+k'},
\end{equation}
and the Jacobi imaginary transformations (see, e.g., \cite{jacobi1829})
\begin{equation}\label{eq:jacobi_transke}
    K(ik) = \frac{1}{\sqrt{1+k^2}}K\left(\frac{k}{\sqrt{1+k^2}}\right),\hspace{.5cm}E(ik)=\sqrt{1+k^2}E\left(\frac{k}{\sqrt{1+k^2}}\right).
\end{equation}
Hold strictly throughout the entire complex plane for the variable $k$, subject to the sole condition that the resulting modulus does not encounter the standard branch points, which implies $k^2 \neq -1$. When restricted to the real domain, the equality is satisfied unconditionally for all $k \in \mathbb{R}$.
Replacing  $k \to \frac{k}{k'}$ in \eqref{eq:jacobi_transke} we have
\begin{equation}\label{eq:jacobi_transke2}
K\left(\frac{ik}{k'}\right) = k'K(k),\hspace{.5cm}E\left(\frac{ik}{k'}\right) = \frac{1}{k'} E(k),
\end{equation}
with $k \in (0, 1)$. In this domain, the complementary modulus $k'$ remains real-valued and bounded such that $0 < k' < 1$, ensuring that the formulas map consistently to real outputs.
For a modulus $k \in \mathbb{T} = \{k \in \mathbb{C} : \Re(k) > 0, \;\Im(k) > 0\} \cup (0, 1)$, let $k' = \sqrt{1 - k^2}$ denote the complementary modulus defined on its principal branch. Jacobi's reciprocal modulus transformation is given on the principal branch by
\begin{equation}\label{eq:jacobireciprocal}
K\left(\frac{1}{k}\right) = k \left( K(k) - i K(k') \right).
\end{equation}
We also recall Clausen's \cite{clausen1828} classical identity for the square of the first complete elliptic integral
\begin{equation}\label{eq:cl_1_2k}
\begin{aligned}
K^2(k) &= \frac{\pi^2}{4} \left[\,_2F_1\left(\frac{1}{2}, \frac{1}{2}; 1; k^2\right)\right]^2 = \frac{\pi^2}{4} \,_3F_2\left(\frac{1}{2}, \frac{1}{2}, \frac{1}{2}; 1, 1; 4k^2(1-k^2)\right) \\
&= \frac{\pi^2}{4}\sum_{m=0}^{\infty}\frac{(2m)!^3}{2^{4m}m!^6}\big(k^2(1-k^2)\big)^{m}, \quad 0 < k \le \frac{1}{\sqrt{2}}.
\end{aligned}
\end{equation}
And its counterpart for the product of both integrals
\begin{equation}\label{eq:cl_1_2e}
E(k)K(k)=\frac{\pi^2}{4}\sum_{m=0}^{\infty}\frac{(2m)!^3\left[m(1-2k^2)+(1-k^2)\right]}{2^{4m}m!^6}\big(k^2(1-k^2)\big)^{m},
\end{equation}
for $0<k<\frac{1}{\sqrt{2}}$. It is obtained from equation \eqref{eq:cl_1_2k} after differentiating with respect to $k$ and applying the derivative formula $\frac{\mathrm{d}K}{\mathrm{d}k}$ given in \eqref{eq:derke} and then isolating the product of $E(k)K(k)$.
We have to note that performing the transformation $k \to \frac{ik}{k'}$ on \eqref{eq:cl_1_2k} and using \eqref{eq:jacobi_transke2} for the first-kind integral gives also
\begin{equation}\label{eq:cl_1_2ki}
\begin{aligned}
K^2(k) &= \frac{\pi^2}{4(1-k^2)}\sum_{m=0}^{\infty}\frac{(-1)^m(2m)!^3}{2^{4m}m!^6}\left(\frac{k}{k^2-1}\right)^{2m},
\end{aligned}
\end{equation}
for  $0 < k \le \sqrt{2}-1$. Similarly, \eqref{eq:cl_1_2e} can be transformed using \eqref{eq:jacobi_transke2} for both first-kind and second-kind integrals
\begin{equation}\label{eq:cl_1_2ei}
E(k)K(k)=\frac{\pi^2}{4(1-k^2)}\sum_{m=0}^{\infty}\frac{(-1)^m(2m)!^3\left[mk^2+m+1\right]}{2^{4m}m!^6}\left(\frac{k}{k^2-1}\right)^{2m},
\end{equation}
for  $0 < k \le \sqrt{2}-1$. 

\section{Explicit Evaluation of the Singular Moduli $k_r$ ($r = 2, 3, 4, 7$) via Fixed-Point Conditions in the Complex Plane}
\begin{lemma}\label{lem:monotonicity}
The period ratio function $f:(0,1)\to(0,\infty)$ defined by $f(k)=\frac{K(k')}{K(k)}$, where $k'=\sqrt{1-k^2}$, is strictly decreasing and bijective. Consequently, for any given $r>0$, there exists a unique modulus $k_r\in(0,1)$ satisfying $f(k_r)=\sqrt{r}$.
\end{lemma}
\begin{proof}
From the power series representation \eqref{eq:K_expansion}, $K(k)$ is positive and strictly increasing on $(0,1)$. Since $k'=\sqrt{1-k^2}$ is strictly decreasing, $K(k')$ is strictly decreasing, making $f(k)=\frac{K(k')}{K(k)}$ strictly decreasing on $(0,1)$. By asymptotic behavior at the boundaries ($\lim_{k\to 0^+} f(k)=+\infty$ and $\lim_{k\to 1^-} f(k)=0$), $f$ is a bijection onto $(0,\infty)$, ensuring the uniqueness of $k_r$.
\end{proof}

\begin{theorem}\label{th:elliptic_relation_theorem}
The real singular moduli $k_r \in (0,1)$ uniquely satisfying the classical complex multiplication relation governed by the elliptic integral ratio
\begin{equation*}
\frac{K(k'_{r})}{K(k_{r})} = \sqrt{r}, \quad \text{for } r \in \{2, 3, 4, 7\},
\end{equation*}
are given explicitly by
\begin{align}
k_{2}  &=\sqrt{2}-1, & k_{3}  &= \frac{\sqrt{2}}{4} (\sqrt{3}-1), \label{eq:singularmoduli23}\\
k_{4}  &=3-2\sqrt{2}, & k_{7}  &= \frac{\sqrt{2}}{8} (3 - \sqrt{7}). \label{eq:singularmoduli47}
\end{align}
\end{theorem}
The singular moduli $k_{2}$, $k_{3}$, and $k_{4}$ were originally calculated by Legendre \cite{legendre1825}, whereas $k_{7}$ follows from the seventh-order modular equation first established by Guetzlaff \cite{guetzlaff1834} (and later systematically tabulated alongside higher-order class invariants by Weber \cite{weber1895}). Here, we provide a completely self-contained derivation of all four values relying exclusively on elementary transformations.
\begin{proof}
\textbf{Case $r = 2.$} 
Let us invoke Landen's ascending transformation \eqref{eq:landeupwardke} and replace $k'$ by $k$ for the complete elliptic integral of the first kind
\begin{equation}
K\left(\frac{1-k}{1+k}\right) = \frac{1+k}{2}K(k'). \label{eq:landenascending}
\end{equation}
By imposing the fixed-point condition under the homographic mapping $\frac{1-k}{1+k} = k$, gives $k=k_{2}=\sqrt{2}-1$ as the real number in $(0,1)$ satisfying this fixed point condition. Substituting $k_{2}$ into \eqref{eq:landenascending} directly yields the desired period ratio
\begin{equation*}
K(k_{2}) = \frac{\sqrt{2}}{2}K(k'_{2}) \implies \frac{K(k'_{2})}{K(k_{2})} = \sqrt{2}.
\end{equation*}

\medskip
\textbf{Case $r = 3.$} 
Let us invoke Jacobi's imaginary and reciprocal transformations for the complete elliptic integral of the first kind (\eqref{eq:jacobi_transke2} and \eqref{eq:jacobireciprocal}). Throughout the following derivations, we take the principal branches for all square roots as well as for the Landen transformations.

Since $K(ik/k')=K(-ik/k')$, imposing the fixed-point condition under the mapping $\frac{1}{k} = -\frac{ik}{k'}$ yields $k=s_{3} = e^{i\pi/6} = \frac{\sqrt{3}}{2} + \frac{i}{2}$ as the unique value in $\mathbb{T}$ satisfying this relation. We thus have
\begin{equation}
K(s_{3})\overset{\eqref{eq:landeupwardke}}=\frac{2}{1+s'_{3}}K\left(\frac{1-s'_{3}}{1+s'_{3}}\right)\overset{\eqref{eq:jacobi_transke}}=(k'_{3}+ik_{3})K(k_{3}),\label{eq:K_s3_eval}
\end{equation}
where $k_3 = \frac{\sqrt{2}}{4}(\sqrt{3}-1)$ and $k'_{3}= \frac{\sqrt{2}}{4}(\sqrt{3}+1)$ are both in $(0,1)$. Furthermore
\begin{equation}
K(s'_{3})\overset{\eqref{eq:landeupwardke}}=\frac{2}{1+s_{3}}K\left(\frac{1-s_{3}}{1+s_{3}}\right)\overset{\eqref{eq:jacobi_transke}}=(k'_{3}-ik_{3})K(k_{3}).\label{eq:K_scomp3_eval}
\end{equation}
Applying the downward Landen transformation \eqref{eq:landendownwardke} (with its principal branch) to $s'_{3}$, we obtain
\begin{equation*}
\frac{2\sqrt{s'_{3}}}{1+s'_{3}} = \frac{2e^{-i\pi/12}}{1+e^{-i\pi/6}} = \frac{1}{\cos(\pi/12)} = \frac{1}{k'_3},
\end{equation*}
and consequently
\begin{equation*}
K\left(\frac{1}{k'_3}\right)=(1+s'_{3})K(s'_{3}) \overset{\eqref{eq:K_scomp3_eval}}=k'_{3}(\sqrt{3}-i)K(k_{3}).
\end{equation*}
On the other hand, since $k'_3 \in  \mathbb{T}$, applying \eqref{eq:jacobireciprocal} directly to $k'_3$ gives
\begin{equation*}
K\left(\frac{1}{k'_3}\right) = k'_3 \left( K(k'_3) - i K(k_3) \right).
\end{equation*}
Equating these last two identities yields
\begin{equation*}
\frac{K(k'_3)}{K(k_3)} = \sqrt{3}.
\end{equation*}

\medskip
\textbf{Case $r = 4.$} 
The first singular modulus is obtained by imposing the fixed-point condition under the mapping $k'=k$, then it follows immediately that  $k_{1}=\frac{1}{\sqrt{2}}$ and $K(k'_{1})/K(k_{1})=1$. Recalling Landen's ascending transformation
\begin{equation*}
K\left(\frac{1-k'}{1+k'}\right) = \frac{1+k'}{2}K(k). 
\end{equation*}
Setting the modulus  $k_1 = \frac{1}{\sqrt{2}}$ in this last formula yields $k_{4}=\frac{1-k'_{1}}{1+k'_{1}}=3-2\sqrt{2}$ and
\begin{equation*}
K(k_{4}) =\left(\frac{\sqrt{2}}{4}+\frac{1}{2} \right)K(k_1).
\end{equation*}
Using Eq.~\eqref{eq:landenascending}, we have $\frac{1-k_{4}}{1+k_{4}}=k_{1}$ and
\begin{equation*}
K(k'_{4})= \left(\frac{\sqrt{2}}{2}+1\right)K(k_{1}).
\end{equation*}
Combining these last two identities directly yields the period ratio
\begin{equation*}
\frac{K(k'_{4})}{K(k_{4})} = 2.
\end{equation*}
\medskip

\textbf{Case $r = 7.$} 
Let us invoke Jacobi's imaginary transformation and Landen's ascending transformation for the first-kind integral (\eqref{eq:jacobi_transke2} and \eqref{eq:landenascending}). Throughout the following derivations, we take the principal branches for all square roots as well as for the Landen transformations. Since $K(ik/k')=K(-ik/k')$, imposing the fixed-point condition under the mapping $\frac{1-k}{1+k} = -\frac{ik}{k'}$ gives $k=s_{7} = \frac{3}{8} + \frac{\sqrt{7}i}{8}$ as the unique value in $\mathbb{T} $ satisfying this fixed point condition. 
We have
\begin{equation}
K(s_{7})\overset{\eqref{eq:landeupwardke}}=\frac{2}{1+s'_{7}}K(\frac{1-s'_{7}}{1+s'_{7}})\overset{\eqref{eq:jacobi_transke}}=(k'_{7}+ik_{7})K(k_{7})\label{eq:K_s7_eval},
\end{equation}
where $k_7 = \frac{\sqrt{2}}{8}(3-\sqrt{7})$ and $k'_{7}= \frac{\sqrt{2}}{8}(3+\sqrt{7})$ are both in $(0,1)$. Furthermore
\begin{align}
&K(s'_{7})\overset{\eqref{eq:landeupwardke}}=\frac{2}{1+s_{7}}K(\frac{1-s_{7}}{1+s_{7}})=\left(\frac{11}{8}-\frac{\sqrt{7}i}{8} \right)K(\frac{3}{8}-\frac{\sqrt{7}i}{8}) \nonumber \\
&\overset{\eqref{eq:landeupwardke}}=\left(\frac{11}{8}-\frac{\sqrt{7}i}{8} \right)(1+i(3\sqrt{7}-8))K(i(3\sqrt{7}-8)) \overset{\eqref{eq:jacobi_transke}}=((4-\sqrt{7})k'_{7}-i(4+\sqrt{7})k_{7})K(k_{7}).\label{eq:K_scomp7_eval}
\end{align}
Applying the downward Landen transformation \eqref{eq:landendownwardke} (with its principal branch) to $s'_{7}$, we obtain
\begin{equation*}
\frac{2\sqrt{s'_{7}}}{1+s'_{7}} =6\sqrt{2}-2\sqrt{14}= \frac{1}{k'_7},
\end{equation*}
and then
\begin{equation*}
K\left(\frac{1}{k'_7}\right)=(1+s'_{7})K(s'_{7}) \overset{\eqref{eq:K_scomp7_eval}}=k'_{7}(\sqrt{7}-i)K(k_{7}).
\end{equation*}
On the other hand, since $k'_7 \in  \mathbb{T}$, applying \eqref{eq:jacobireciprocal} directly to $k'_7$ gives
\begin{equation*}
K\left(\frac{1}{k'_7}\right) = k'_7 \left( K(k'_7) - i K(k_7) \right).
\end{equation*}
Equating these last two identities yields
\begin{equation*}
\frac{K(k'_7)}{K(k_7)} = \sqrt{7}.
\end{equation*}
\end{proof}

\begin{remark}
Similar fixed-point considerations apply to $r = 15, 31,$ and $63$. More generally, it can be shown that for any $p = 2^n - 1$, the singular modulus $k_p$ is constructible using the Landen and Jacobi transformations; however, the details of this general construction lie beyond the scope of this paper.
\end{remark}

\section{Explicit Evaluations of Ramanujan Series in the Theory of Signature 2}
\begin{lemma}\label{lem:elliptic_relation_lemma}
Let $k \in \mathbb{C}$ be an elliptic modulus satisfying $\Re(k) > 0$ ($k \neq 1$). By the Landen transformation of degree $2$, the complete elliptic integrals satisfy the identity
\begin{equation}\label{lem:elliptic_relation}
\frac{K(k')}{K(k)}\left(K^2(k)-\frac{K(k)E(k)}{1+k}\right) = - \frac{\pi}{2(1+k)} + K(k)E\left(\frac{1-k}{1+k}\right),
\end{equation}
where the complementary modulus $k' = \sqrt{1-k^2}$ and its transformed counterpart are evaluated on their principal branches to avoid sheet-transition ambiguities.
\end{lemma}
\begin{proof}
We first establish the identity for real $k \in (0, 1)$, where all moduli lie strictly within $(0,1)$ and the complete elliptic integrals take real, unambiguous values. Let
\begin{align}
z &= \frac{2\sqrt{k}}{1+k}, \label{eq:landen_k1_relation} \\
z' &= \sqrt{1-z^2} = \frac{1-k}{1+k} \label{eq:landen_modulus_complementary}
\end{align}
denote the transformed modulus and its complementary counterpart. Applying the Legendre relation \eqref{eq:legendre_relation} to $z \in (0,1)$ and substituting the Landen transformations \eqref{eq:landendownwardke} for $K(z)$ and $E(z)$, together with $K(z') = \frac{1+k}{2}K(k')$, yields \eqref{lem:elliptic_relation} for $k \in (0, 1)$ after straightforward algebraic simplification.

To extend \eqref{lem:elliptic_relation} to $\Re(k) > 0$, observe that \eqref{eq:landen_modulus_complementary} maps $\{k \in \mathbb{C} : \Re(k) > 0\}$ conformally onto $\{z' \in \mathbb{C} : |z'| < 1\}$, avoiding the branch cuts $(-\infty, -1] \cup [1, \infty)$ of $E(\cdot)$. Since $k' = \sqrt{1-k^2}$, $K(k)$, and $E(k)$ are holomorphic on $\{k \in \mathbb{C} : \Re(k) > 0\} \setminus [1, \infty)$ with $K(k) \neq 0$ under their principal branches, the identity theorem extends \eqref{lem:elliptic_relation} uniquely from $(0,1)$ to $\Re(k) > 0$ (and by continuity onto $(1, \infty)$).

Finally, on $|k| = 1$ with $\Re(k) \geq 0$ ($k \neq 1$), \eqref{eq:landen_modulus_complementary} sends $k = e^{i\theta}$ ($\theta \in [-\pi/2, \pi/2] \setminus \{0\}$) onto the imaginary axis with $(z')^2 = -\tan^2(\theta/2) \in [-1, 0)$, bounded away from $[1, \infty)$. Thus, the hypergeometric representations
\begin{equation*}
K(z') = \frac{\pi}{2} \, {}_2F_1\left(\frac{1}{2}, \frac{1}{2}; 1;(z')^2\right), \qquad E(z') = \frac{\pi}{2} \, {}_2F_1\left(-\frac{1}{2}, \frac{1}{2}; 1;(z')^2\right)
\end{equation*}
remain analytic across $|k| = 1$ (including $k = \pm i$), and taking the radial limit $\lim_{|k| \to 1^-} \left[ K(z)E(z') + K(z')E(z) - K(z)K(z') \right] = \frac{\pi}{2}$ confirms that \eqref{lem:elliptic_relation} persists smoothly without sheet-transition ambiguities under the principal branches.
\end{proof}

\begin{theorem}
The classical series representations for $1/\pi$ are valid
\begin{equation}\label{eq:bauer}
\sum_{m=0}^{\infty}\frac{(-1)^m(2m)!^3}{2^{6m}m!^6}(4m+1)=\frac{2}{\pi},
\end{equation}
\begin{equation}\label{eq:ramanujank3}
\sum_{m=0}^{\infty} \frac{(2m)!^3}{2^{8m} m!^6} (6m+1) = \frac{4}{\pi},
\end{equation}
\begin{equation}\label{eq:berndtetal}
\sum_{m=0}^{\infty}\frac{(-1)^m(2m)!^3}{2^{9m}m!^6}(6m+1)=\frac{2\sqrt{2}}{\pi},
\end{equation}
\begin{equation}\label{eq:ramanujank7}
\sum_{m=0}^{\infty} \frac{(2m)!^3}{2^{12m} m!^6} (42m+5) = \frac{16}{\pi}.
\end{equation}
\end{theorem}
\begin{proof}
\noindent\textbf{Proof of \eqref{eq:bauer}.}
Using $k_{2}$ from \eqref{eq:singularmoduli23}, we have seen before that $k_2=\frac{1-k_2}{1+k_2}$ and $K(k'_{2})/K(k_{2})=\sqrt{2}$, using \eqref{lem:elliptic_relation} yields
\begin{align}\label{eq:rel_n2}
\sqrt{2}\left(K^2(k_2)-\frac{K(k_2)E(k_2)}{\sqrt{2}} \right)=-\frac{\pi}{2\sqrt{2}}+K(k_2)E(k_2) \nonumber
\\ 
\implies 8E(k_2)K(k_2) - 4\sqrt{2}K^2(k_2) = \sqrt{2}\pi.
\end{align}
Using Eqs.~\eqref{eq:cl_1_2ki}, \eqref{eq:cl_1_2ei} and \eqref{eq:rel_n2} then
\begin{align*}
\sqrt{2}\pi=8E(k_2)K(k_2) - 4\sqrt{2}K^2(k_2) &=  \frac{\sqrt{2}\pi^2}{2}\sum_{m=0}^{\infty}\frac{(-1)^m(2m)!^3}{2^{6m}m!^6}(4m+1).
\end{align*}

\medskip
\noindent\textbf{Proof of \eqref{eq:ramanujank3}.}
For this case we use $k_{3}=\frac{\sqrt{2}}{4}(\sqrt{3}-1)$ from \eqref{eq:singularmoduli23}. Let $s_3 = i^{1/3} = \frac{\sqrt{3}}{2} + \frac{i}{2}$. We begin by establishing the following identity connecting the complete elliptic integrals evaluated at $s_3$ and $k_3$
\begin{equation}
E(s_{3})K(s_{3})=2E(k_{3})K(k_{3})-K^2(k_{3}). \label{eq:s3_identity}
\end{equation}
Utilizing Eq.~\eqref{eq:landeupwardke} in conjunction with Eq.~\eqref{eq:jacobi_transke}, the second-kind integral expands as
\begin{align}
E(s_3) &= \left(1+\frac{\sqrt{3}}{2}-\frac{i}{2}\right)E(i(2-\sqrt{3})) - \left(\frac{\sqrt{3}}{2}-\frac{i}{2}\right)K(s_3) \nonumber \\
&=\left(1+\frac{\sqrt{3}}{2}-\frac{i}{2}\right)\left(\sqrt{6}-\sqrt{2}\right)E(k_3) - \left(\frac{\sqrt{3}}{2}-\frac{i}{2}\right)K(s_3) \nonumber \\
&\overset{\eqref{eq:K_s3_eval}} =2\left(k'_3-ik_3 \right)E(k_3)+\left(-k'_3+ik_3 \right)K(k_3). \label{eq:E_s3_eval}
\end{align}
Direct algebraic multiplication of Eqs.~\eqref{eq:K_s3_eval} and \eqref{eq:E_s3_eval} confirms the validity of the relation in Eq.~\eqref{eq:s3_identity}.
Evaluating \eqref{lem:elliptic_relation} at $s_3 = i^{1/3}$ yields
\begin{equation}
\left(\frac{\sqrt{3}}{2}-\frac{i}{2}\right)\left(K^2(s_3) - \frac{K(s_3)E(s_3)}{1+s_3}\right) = -\frac{\pi}{2(s_3+1)} + K(s_3)E(i(\sqrt{3}-2)), \label{eq:lemma_evaluated}
\end{equation}
where we have utilized Eqs.~\eqref{eq:K_s3_eval} and \eqref{eq:K_scomp3_eval} to get $K(s'_{3})/K(s_{3})=\frac{\sqrt{3}}{2}-\frac{i}{2}$. Substituting the uncoupled identity \eqref{eq:s3_identity} into Eq.~\eqref{eq:lemma_evaluated} gives
\begin{equation}
\left(\frac{\sqrt{3}}{2}-\frac{i}{2}\right)\left(K^2(s_3) - \frac{2E(k_3)K(k_3) - K^2(k_3)}{1+s_3}\right) = -\frac{\pi}{2(s_3+1)} + K(s_3)E(i(\sqrt{3}-2)). \label{eq:lemma_substituted}
\end{equation}
Applying the imaginary transformation for the second-kind integral from Eq.~\eqref{eq:jacobi_transke} the remaining second-kind integral simplifies to
\begin{equation}
E(i(\sqrt{3}-2))=4k_3E(k_3). \label{eq:E_imag_final}
\end{equation}
Finally, using the value of $s_{3}$, combining Eqs.~\eqref{eq:K_s3_eval}, \eqref{eq:lemma_substituted}, and \eqref{eq:E_imag_final},  collecting the terms, we arrive at the complex algebraic balance
\begin{align*}
K^2(k_3)&\left(\frac{\sqrt{3}}{2}+\frac{1}{2}\right) + (1-\sqrt{3})E(k_3)K(k_3) + i\left[K^2(k_3)\left(\frac{1}{2}-\frac{\sqrt{3}}{2}\right) + (\sqrt{3}-1)E(k_3)K(k_3)\right] \\
&= E(k_3)K(k_3) - \frac{\pi}{4} + i\left[(2-\sqrt{3})E(k_3)K(k_3) + \pi\left(\frac{1}{2}-\frac{\sqrt{3}}{4}\right)\right].
\end{align*}
By equating the real parts of both sides of this last equality (the imaginary parts yield an identical constraint), we obtain
\begin{align}\label{eq:rel_n3}
12E(k_3)K(k_3) - (2\sqrt{3}+6)K^2(k_3) = \sqrt{3}\pi.
\end{align}
Using Eqs.~\eqref{eq:cl_1_2k}, \eqref{eq:cl_1_2e} and \eqref{eq:rel_n3} gives
\begin{align*}
\sqrt{3}\pi = 12E(k_3)K(k_3) - (2\sqrt{3}+6)K^2(k_3) &= \frac{\sqrt{3}\pi^2}{4} \sum_{m=0}^{\infty} \frac{(2m)!^3}{2^{8m} m!^6} (6m+1).
\end{align*}

\medskip
\noindent\textbf{Proof of \eqref{eq:berndtetal}.}
For this case we use $k_{4}=3-2\sqrt{2}$ from \eqref{eq:singularmoduli47}. Since we have seen that $K(k'_4)/K(k_4)=2$ and $k_1=\frac{1-k_4}{1+k_4}$, then using \eqref{lem:elliptic_relation} yields
\begin{equation}\label{eq:eqcase4}
2\left(K^2(k_4)-\frac{K(k_4)E(k_4)}{4-2\sqrt{2}} \right)=-\frac{\pi}{2(4-2\sqrt{2})}+K(k_4)E(k_1), 
\end{equation}
We invoke Landen ascending transformation \eqref{eq:landeupwardke} for the second-kind integral and then
\begin{equation*}
E(\frac{1-k'_1}{1+k'_1})=E(k_4)=\frac{E(k_1)}{1+k'_1}+\frac{k'_1K(k_1)}{1+k'_1},
\end{equation*}
Also, we have seen that $K(k_4)=(\frac{\sqrt{2}}{4}+\frac{1}{2})K(k_1)$ and substituting into the last equation yields
\begin{equation}\label{eq:eqcase4decoupled}
E(k_1)=(2-2\sqrt{2})K(k_4)+(1+\frac{1}{\sqrt{2}})E(k_4)
\end{equation}
Replacing \eqref{eq:eqcase4decoupled} into \eqref{eq:eqcase4} yields
\begin{align}\label{eq:rel_n4}
16E(k_4)K(k_4) - 32(\sqrt{2}-1)K^2(k_4) = 2\pi.
\end{align}
Using Eqs.~\eqref{eq:cl_1_2ki}, \eqref{eq:cl_1_2ei} and \eqref{eq:rel_n4} gives
\begin{align*}
2\pi=16E(k_4)K(k_4) - 32(\sqrt{2}-1)K^2(k_4) &= \frac{\pi^2}{\sqrt{2}}\sum_{m=0}^{\infty}\frac{(-1)^m(2m)!^3}{2^{9m}m!^6}(6m+1).
\end{align*}

\medskip
\noindent\textbf{Proof of \eqref{eq:ramanujank7}.}
For this case, we use $k_{7}=\frac{\sqrt{2}}{8}(3-\sqrt{7})$ from \eqref{eq:singularmoduli47}.
Let $s_7 = \frac{3}{8} + \frac{\sqrt{7}i}{8}$. We begin by establishing the following algebraic relation between the complete elliptic integrals
\begin{equation}
E(s_7)K(s_7) = 2E(k_7)K(k_7) - K^2(k_7). \label{eq:s7_base_identity}
\end{equation}
\begin{samepage}
The second-kind integral can be expanded by applying the corresponding transformations recursively
\begin{align}
E(s_7) &\overset{\eqref{eq:landeupwardke}}{=} \left(1+\frac{3\sqrt{7}}{8}-\frac{i}{8}\right)E(i(8-3\sqrt{7})) - \left(\frac{3\sqrt{7}}{8}-\frac{i}{8}\right)K\left(s_7\right) \nonumber \\
&\overset{\eqref{eq:jacobi_transke}}{=} \left(1+\frac{3\sqrt{7}}{8}-\frac{i}{8}\right)(6\sqrt{2}-2\sqrt{14})E(k_7) - \left(\frac{3\sqrt{7}}{8}-\frac{i}{8}\right)K\left(s_7\right) \nonumber \\
&\overset{\eqref{eq:K_s7_eval}}{=}2(k'_7-ik_7)E(k_7)-(k'_7-ik_7)K(k_7).\label{eq:E_s7_eval_7}
\end{align}
\end{samepage}
Direct algebraic multiplication of Eqs.~\eqref{eq:K_s7_eval} and \eqref{eq:E_s7_eval_7} confirms the validity of the constraint in Eq.~\eqref{eq:s7_base_identity}. 
Evaluating \eqref{lem:elliptic_relation} at the complex modulus $s_7$ and substituting the relations from Eq.~\eqref{eq:s7_base_identity} we obtain 
\begin{equation}
\left(\frac{\sqrt{7}}{2}-\frac{i}{2}\right)\left(K^2(s_7) - \frac{2E(k_7)K(k_7)-K^2(k_7)}{1+s_7}\right) = -\frac{\pi}{2(1+s_7)} + K(s_7)E\left(\frac{1-s_7}{1+s_7}\right), \label{eq:lemma_sub_7}
\end{equation}
where we have utilized Eqs.~\eqref{eq:K_s7_eval} and \eqref{eq:K_scomp7_eval} to get $K(s'_{7})/K(s_{7})=\frac{\sqrt{7}}{2}-\frac{i}{2}$.
On the other hand, the remaining transformed second-kind complete elliptic integral on the right-hand side can be carefully decoupled.
Applying \eqref{eq:landeupwardke} consecutively, followed by the imaginary transformations in \eqref{eq:jacobi_transke}, the complete elliptic integral of the second kind evaluates directly to
\begin{align}
E\left(\frac{1-s_7}{1+s_7}\right) &= E\left(\frac{3}{8}-\frac{\sqrt{7}i}{8}\right) \notag \\ 
\MoveEqLeft[6] =\left(\frac{3\sqrt{7}}{8}+1+\frac{i}{8} \right)E(i(3\sqrt{7}-8))-\left(\frac{3\sqrt{7}}{8}+\frac{i}{8} \ \right)K(\frac{3}{8}-\frac{\sqrt{7}i}{8}) \notag \\ 
\MoveEqLeft[6] = 2(k'_7 +i k_7)E(k_7) - (k'_7 + i k_7)K(k_7), \label{eq:E_landen_decoupled_7}
\end{align}
where in the last step we have used  $K(\frac{3}{8}-\frac{\sqrt{7}i}{8}) =(k'_{7}-ik_{7})K(k_{7})$ which follows from \eqref{eq:K_scomp7_eval}. Finally, using the value of $s_7$, substituting the explicit evaluations of Eqs.~\eqref{eq:K_s7_eval} and \eqref{eq:E_landen_decoupled_7} into the structured relation from Eq.~\eqref{eq:lemma_sub_7} results in the comprehensive complex equality
\begin{align*}
\left(\frac{5\sqrt{7}}{16}+\frac{11}{8}\right)K^2(k_7) &- \frac{5\sqrt{7}K(k_7)E(k_7)}{8} + i\left[\frac{9K(k_7)E(k_7)}{8}-\left(\frac{\sqrt{7}}{8}+\frac{9}{16}\right)K^2(k_7)\right] \\
\MoveEqLeft[6] = -\frac{3\sqrt{7}K^2(k_7)}{8} + \frac{3\sqrt{7}K(k_7)E(k_7)}{4} - \frac{11\pi}{32}  + i\left[-\frac{K^2(k_7)}{8} +\frac{K(k_7)E(k_7)}{4} + \frac{\sqrt{7}\pi}{32}\right].
\end{align*}
By equating the real parts or imaginary parts of both sides of this last expression, we obtain the isolated real balance equation
\begin{equation}\label{eq:rel_n7}
28E(k_7)K(k_7) - (4\sqrt{7}+14)K^2(k_7) = \sqrt{7}\pi.
\end{equation}
Using Eqs.~\eqref{eq:cl_1_2k}, \eqref{eq:cl_1_2e} and \eqref{eq:rel_n7} yields
\begin{align*}
\sqrt{7}\pi = 28E(k_7)K(k_7) - (4\sqrt{7}+14)K^2(k_7) &= \frac{\sqrt{7}\pi^2}{16} \sum_{m=0}^{\infty} \frac{(2m)!^3}{2^{12m} m!^6} (42m+5).
\end{align*}
\end{proof}The series \eqref{eq:bauer} corresponds to the celebrated formula of Bauer \cite{bauer1859}, whereas \eqref{eq:berndtetal} is due to Berndt et al. \cite{baruahberndt2009eisenstein}. The remaining two relations, \eqref{eq:ramanujank3} and \eqref{eq:ramanujank7}, correspond to formulas (28) and (29) in Ramanujan's paper \cite{ramanujan1914}.

\section{Applications to Harmonic Ramanujan-Type Series in the Theory of Signature 2}
In recent years, Ramanujan-type series for $1/\pi$ involving harmonic numbers have garnered considerable interest. Following early elliptic integral derivations by Campbell \cite{campbell2018}, much of the recent literature has obtained harmonic evaluations through hypergeometric coefficient extraction, differential operators, WZ pairs, and modular parametrizations (see, e.g., \cite{houhewang2024, lichu2024, sunzhou2026}). In the final part of this paper, we derive two companion families of logarithm-free rational harmonic Ramanujan-type series for $1/\pi$ involving higher-degree polynomials.

\begin{proposition}\label{eq:bin_harm}
Let $H_m = \sum_{i=1}^{m}\frac{1}{i}$ be the $m$-th harmonic number.
The following expansion holds true for $0<k<1$
\begin{equation}
\frac{1}{\sqrt{1-k^2}}=\frac{1}{k^2}\sum_{m=1}^{\infty}\frac{(2m)!H_{m}}{2^{2m}m!^2}\left[\frac{k^2(m+1)(2m-1)-2m^2}{1-2m}\right]k^{2m}.
\end{equation}
\end{proposition}
\begin{proof}
By employing the classical generating function for the harmonic numbers $H_m$, we have
\begin{equation*}
\frac{\log{(1-x)}}{x-1}=\sum_{m=1}^{\infty}H_{m}x^{m}\quad \text{for } x \in (0,1).
\end{equation*}
Multiplying both sides by $(x-1)$ and taking the derivative with respect to $x$, we obtain
\begin{equation}\label{eq:harm_der}
\frac{1}{x-1} = \sum_{m=1}^{\infty} H_{m} x^{m-1} \left[ (m+1)x - m \right].
\end{equation}
Next, we perform the change of variable $x = k^2 \sin^2 t$ in Eq.~\eqref{eq:harm_der}. Integrating both sides with respect to $t$ over the interval $(0, \pi/2)$ yields for the left-hand side
\begin{align*}
I=\int_{0}^{\pi/2} \frac{dt}{k^2 \sin^2 t - 1} \overset{t \rightarrow \arctan{u}}=\int_{0}^{\infty}\frac{du}{u^2(k^2-1)-1}\\=-\frac{1}{\sqrt{1-k^2}} \left[ \arctan\left(u\sqrt{1-k^2}\right) \right]_{0}^{\infty}=-\frac{\pi}{2\sqrt{1-k^2}}.
\end{align*}
For the right-hand side, invoking the Dominated Convergence Theorem (DCT) justifies the interchange of integration and summation for $0 < k < 1$
\begin{align*}
I=\int_{0}^{\pi/2}\left(\sum_{m=1}^{\infty} H_{m} k^{2m-2}\sin^{2m}{t} \left[ (m+1)k^2\sin^2{t}- m \right]\right)dt \\
=\sum_{m=1}^{\infty}H_{m}k^{2m-2}\left[(m+1)k^2\int_{0}^{\pi/2}\sin^{2m+2}{t}dt -m\int_{0}^{\pi/2}\sin^{2m}{t}dt\right] \\
=\frac{\pi}{2k^2}\sum_{m=1}^{\infty}\frac{(2m)!H_{m}}{2^{2m}m!^2}\left[ \frac{k^2(m+1)(2m-1)-2m^2}{2m-1} \right]k^{2m}, 
\end{align*}
where in the last step we have used Wallis' formula 
\begin{equation}\label{eq:wallis}
\int_{0}^{\pi/2}\sin^{2m}{t}dt=\frac{\pi}{2}\frac{(2m)!}{2^{2m}m!^2}, \quad \quad m \in \mathbb{N}.
\end{equation}
\end{proof}

\begin{proposition}
Let $H_m = \sum_{i=1}^{m}\frac{1}{i}$ be the $m$-th harmonic number.
The following expansion for the complete elliptic integral of the first kind holds true for $0<k<1$
\begin{equation}\label{eq:bin_harm_k}
K(k) = \frac{\pi}{2k^2}\sum_{m=1}^{\infty}\frac{(2m)!^2H_{m}}{2^{4m}m!^4}\left[\frac{4m^3-k^2(m+1)(2m-1)^2}{(2m-1)^{2}}\right]k^{2m}.
\end{equation}
\end{proposition}
\begin{proof}
Applying the substitution $k \to k \sin t$ in Eq.~\eqref{eq:bin_harm} and integrating with respect to $t$ over the interval $(0, \pi/2)$ yields 
\begin{align*}
K(k)=\int_{0}^{\pi/2}\left(\sum_{m=1}^{\infty}\frac{(2m)!H_{m}}{2^{2m}m!^2}\left[\frac{(k\sin{t})^2(m+1)(2m-1)-2m^2}{1-2m}\right](k\sin{t})^{2m-2}\right)dt \\ 
\overset{DCT}=\sum_{m=1}^{\infty}\frac{(2m)!H_{m}}{2^{2m}m!^2}\int_{0}^{\pi/2}\left(\left[\frac{(k\sin{t})^2(m+1)(2m-1)-2m^2}{1-2m}\right](k\sin{t})^{2m-2}\right) dt \\
\overset{\eqref{eq:wallis}}=\frac{\pi}{2k^2}\sum_{m=1}^{\infty}\frac{(2m)!^2H_{m}}{2^{4m}m!^4}\left[\frac{4m^3-k^2(m+1)(2m-1)^2}{(2m-1)^{2}}\right]k^{2m}.
\end{align*}
\end{proof}

\begin{lemma}
The following integral identity holds true for $k \in (0,1]$
\begin{equation}\label{eq:integral_clausen_identity}
\int_{0}^{\pi/2}K(k\sin{\alpha})\,d\alpha=K^2\left(\sqrt{\frac{1}{2}\left(1-\sqrt{1-k^2}\right)}\right).
\end{equation}
\end{lemma}

\begin{proof}
Applying the substitution $k \to k\sin{\alpha}$ in Eq.~\eqref{eq:K_expansion}, integrating over $(0,\pi/2)$ with respect to $\alpha$, and employing Wallis' formula \eqref{eq:wallis} yields 
\begin{align}\label{eq:integral_identity}
\int_{0}^{\pi/2}K(k\sin{\alpha})\,d\alpha=\frac{\pi^2}{4}\sum_{m=0}^{\infty}\frac{(2m)!^3}{2^{6m}m!^6}k^{2m}, \quad k \in (0,1).
\end{align}
By virtue of Abel's limit theorem for power series and the Monotone Convergence Theorem, both sides converge to the same finite limit as $k \to 1^-$. Consequently, the identity \eqref{eq:integral_identity} extends continuously to the closed boundary $k=1$. 
Next, we establish the connection with the square of the complete elliptic integral. Utilizing Clausen's formula \eqref{eq:cl_1_2k} with an argument $s$, we recall that the underlying hypergeometric series is strictly valid for $s \in (0,\frac{1}{\sqrt{2}}]$. Under the algebraic transformation $s = \sqrt{\frac{1}{2}\left(1-\sqrt{1-k^2}\right)}$, the interval $k \in (0,1]$ is mapped bijectively onto the required domain $s \in (0,\frac{1}{\sqrt{2}}]$. By performing this change of variable, the argument in Clausen's series naturally expands to $4s^2(1-s^2) = k^2$, which yields
\begin{align*}
K^2\left(\sqrt{\frac{1}{2}\left(1-\sqrt{1-k^2}\right)}\right) &= \frac{\pi^2}{4}\sum_{m=0}^{\infty}\frac{(2m)!^3}{2^{6m}m!^6}k^{2m} \overset{\eqref{eq:integral_identity}}{=} \int_{0}^{\pi/2}K(k\sin{\alpha})\,d\alpha.
\end{align*}
This completes the proof.
\end{proof}

\begin{theorem}\label{th:clausen_harm}
The following closed-form harmonic expansions hold true
\begin{equation}\label{eq:harmonic_clausen_k1}
K^{2}(k) = \frac{\pi^2}{4k^2(1-k^2)} \sum_{m=1}^{\infty} \frac{(2m)!^3 H_{m} P_{m}(k)}{2^{4m} m!^6} \left( k^2(1-k^2) \right)^{m}, \quad 0 < k \le \frac{1}{\sqrt{2}},
\end{equation}
where the polynomial $P_{m}(k)$ is defined by
\begin{equation*}
P_{m}(k):=\frac{2m^4}{(2m-1)^3} - (m+1)k^2(1-k^2).
\end{equation*}
Furthermore, the product of the complete elliptic integrals of the first and second kind admits the expansion
\begin{equation}\label{eq:harmonic_clausen_ke1}
E(k)K(k) = \frac{\pi^2}{4k^2(1-k^2)} \sum_{m=1}^{\infty} \frac{(2m)!^3 H_{m} Q_{m}(k)}{2^{4m} m!^6} \left( k^2(1-k^2) \right)^{m}, \quad 0 < k < \frac{1}{\sqrt{2}},
\end{equation}
where $Q_{m}(k)$ is given by
\begin{align*}
Q_{m}(k):=\frac{2m^5}{(2m-1)^3} - \frac{k^2(6m^4+4m^3-3m^2-2m+1)}{(2m-1)^2} 
 \\  +(m+1)\left[k^4(3m+2) - k^{6}(2m+1)\right].
\end{align*}
\end{theorem}

\begin{proof}
\noindent\textbf{Proof of \eqref{eq:harmonic_clausen_k1}.}
Using Eq.~\eqref{eq:integral_clausen_identity}, applying the substitution $k \to k\sin{\alpha}$ in Eq.~\eqref{eq:bin_harm_k} and integrating over $(0,\pi/2)$ with respect to $\alpha$ yields 
\begin{align*}
&K^2\left(\sqrt{\frac{1}{2}\left(1-\sqrt{1-k^2}\right)}\right)= \int_{0}^{\pi/2} K(k\sin\alpha) \, d\alpha \\[4pt]
&\quad = \int_{0}^{\pi/2} \left( \frac{\pi}{2} \sum_{m=1}^{\infty} \frac{(2m)!^2 H_{m}}{2^{4m} m!^4} \left[ \frac{4m^3 - (k\sin\alpha)^2(m+1)(2m-1)^2}{(2m-1)^2} \right] (k\sin\alpha)^{2m-2} \right) d\alpha \\[4pt]
&\quad \overset{\text{DCT}}{=} \frac{\pi}{2} \sum_{m=1}^{\infty} \frac{(2m)!^2 H_{m}}{2^{4m} m!^4}\int_{0}^{\pi/2} \left(\left[ \frac{4m^3 - (k\sin\alpha)^2(m+1)(2m-1)^2}{(2m-1)^2} \right] (k\sin\alpha)^{2m-2} \right) d\alpha \\[4pt]
&\quad \overset{\eqref{eq:wallis}}{=} \frac{\pi^2}{4k^2} \sum_{m=1}^{\infty} \frac{(2m)!^3 H_{m}}{2^{6m} m!^6}\left[\frac{k^2(m+1)(2m-1)^3-8m^4}{(1-2m)^3} \right]k^{2m}.
\end{align*}
The result follows after making the transformation $k \to 2k\sqrt{1-k^2}$.

\medskip
\noindent\textbf{Proof of \eqref{eq:harmonic_clausen_ke1}.}
Let $R(k)$ denote the right-hand side in Eq.~\eqref{eq:harmonic_clausen_k1}. Differentiating Eq.~\eqref{eq:harmonic_clausen_k1} with respect to $k$ and using \eqref{eq:derke} yields 
\begin{align*}
\frac{2E(k)K(k)}{k(1-k^2)} - \frac{2K^2(k)}{k} = \frac{\mathrm{d}R}{\mathrm{d}k}.
\end{align*}
Then, \eqref{eq:harmonic_clausen_ke1} follows immediately after isolating the product $E(k)K(k)$ and using \eqref{eq:harmonic_clausen_k1}.
Note that identity \eqref{eq:harmonic_clausen_ke1} does not hold at $k = \frac{1}{\sqrt{2}}$ due to the behavior of the mapping $g(k) = k^2(1-k^2)$. 
The function $g(k)$ attains its global maximum of $1/4$ precisely at $k = \frac{1}{\sqrt{2}}$. 
At this critical point, the derivative $\frac{\mathrm{d}g}{\mathrm{d}k} = 2k(1-2k^2)$ vanishes, which causes the algebraic inversion required in the transformation to become non-analytic. 
Consequently, the radius of convergence of the series is strictly bounded, restricting the domain to the open interval $0 < k < \frac{1}{\sqrt{2}}$.

\end{proof}
\begin{corollary}
The following closed-form harmonic expansions hold true 
\begin{equation}\label{eq:harmonic_clausen_k2}
K^{2}(k) = \frac{\pi^2(k^2-1)}{4k^2} \sum_{m=1}^{\infty} \frac{(-1)^{m}(2m)!^3 H_{m} T_{m}(k)}{2^{4m} m!^6}\left(\frac{k}{k^2-1}\right)^{2m},
\end{equation}
for $k \in (0,\sqrt{2}-1]$, where the polynomial $T_{m}(k)$ is defined by
\begin{equation*}
T_{m}(k):=\frac{2m^4}{(2m-1)^3}+\frac{k^2(m+1)}{(k^2-1)^2}.
\end{equation*}
The product of the complete elliptic integrals of the first and second kind admits the expansion
\begin{equation}\label{eq:harmonic_clausen_ke2}
E(k)K(k) =  \frac{\pi^2(1-k^2)^2}{4k^2} \sum_{m=1}^{\infty} \frac{(-1)^{m}(2m)!^3 H_{m} U_{m}(k)}{2^{4m} m!^6}\left(\frac{k}{k^2-1}\right)^{2m},
\end{equation}
for $k \in (0,\sqrt{2}-1)$, where $U_{m}(k)$ is given by
\begin{align*}
U_{m}(k):=\frac{k^4m(6m^3-3m+1)+k^2(4m^3-3m^2-2m+1)+2m^4}{(k^2-1)^3(2m-1)^2}
\\+\frac{2m^4(m-1)}{(2m-1)^3}.
\end{align*}
\end{corollary}
\begin{proof}
Performing the transformation $k \to \frac{ik}{k'}$ in Eqs.~\eqref{eq:harmonic_clausen_k1} and \eqref{eq:harmonic_clausen_ke1} and using Eq.~\eqref{eq:jacobi_transke2} both expansions follow immediately. 

At the boundary $k = \sqrt{2}-1$ the series \eqref{eq:harmonic_clausen_ke2} diverges; here, the ratio test is inconclusive ($L=1$). By applying Stirling's approximation ($\frac{e^{m}m!}{m^{m+1/2}\sqrt{2\pi}} \sim 1$) to the factorials, the combinatorial coefficient scales as $\frac{(2m)!^3}{m!^6} \sim \mathcal{O}\left( \frac{2^{6m}}{m^{3/2}} \right)$. 

Since the exponential factor $2^{6m}$ cancels exactly with the base $2^{4m}$ and the boundary value $(\frac{k}{k^2-1})^{2m} = 2^{2m}$, the remaining algebraic behavior is dictated by the $m^{3/2}$ in the denominator. 

Given that the harmonic numbers scale as $H_m \sim \log(m)$ and the polynomial behaves as $U_{m}(k) \sim \mathcal{O}(m^2)$, the general term of \eqref{eq:harmonic_clausen_ke2} scales asymptotically as
\begin{equation*} 
|a_m| \sim \mathcal{O}\left( \frac{\log(m) \cdot m^2}{m^{3/2}} \right) = \mathcal{O}\left( \log(m)\sqrt{m} \right). 
\end{equation*}
Since $\lim_{m\to\infty} |a_m| = \infty$, the general term fails the necessary condition for convergence ($\lim_{m\to\infty} a_m = 0$). Consequently, the series diverges at the boundary.
\end{proof}
Boyadzhiev proved in \cite{boyadzhiev2012} that
\begin{align}\label{eq:boyadzhiev}
\sum_{m=1}^{\infty}\frac{(2m)!H_{m}}{2^{2m}m!^2}x^{m}=\frac{2\log{(\frac{\sqrt{1-x}+1}{2})}-\log{(1-x)}}{\sqrt{1-x}}, \quad x \in (0,1).
\end{align}
Multiplying both sides of \eqref{eq:boyadzhiev} by $\sqrt{1-x}$ and taking the derivative with respect to $x$ and then multiplying both sides by $x\sqrt{1-x}$
\begin{corollary}
\begin{align}\label{eq:boyadzhiev_bin}
\sum_{m=1}^{\infty}\frac{(2m)!H_{m}}{2^{2m}m!^2}\left[2m-x(2m+1) \right]x^{m}= \frac{2}{\sqrt{1-x}}-2, \quad x \in (0,1).
\end{align}
\end{corollary}
Performing the substitution $x \to k^2\sin^2{t}$ in \eqref{eq:boyadzhiev_bin}, integrating in $(0,\pi/2)$ with respect to $t$ and using Wallis' formula \eqref{eq:wallis} yields 
\begin{corollary}
\begin{equation}\label{eq:bin_harm_k1}
K(k) =\frac{\pi}{2}+\frac{\pi}{8}\sum_{m=1}^{\infty}\frac{(2m)!^2H_{m}}{2^{4m}m!^4}\left[\frac{4m(m+1)-k^2(2m+1)^2}{m+1} \right]k^{2m},\quad  \quad k \in (0,1).
\end{equation}
\end{corollary}
The substitution $k \to k\sin{\alpha}$ in \eqref{eq:bin_harm_k1}, integration in $(0,\pi/2)$ with respect to $\alpha$ and following the same procedure as we have done in Th.~\eqref{th:clausen_harm} gives 
\begin{corollary}
\begin{equation}\label{eq:harmonic_clausen_k3}
K^2(k)=\frac{\pi^2}{4}+\frac{\pi^2}{8}\sum_{m=1}^{\infty}\frac{(2m)!^3H_{m}V_{m}(k)}{2^{4m}m!^6}\left(k^2(1-k^2)\right)^{m},\quad 0<k\le \frac{1}{\sqrt{2}},
\end{equation}
where the polynomial $V_{m}(k)$ is defined by
\begin{equation*}
V_{m}(k):=\frac{(2m+1)^3(k^4-k^2)+2m(m+1)^2}{(m+1)^2}.
\end{equation*}
The product of the complete elliptic integrals of the first and second kind admits the expansion
\begin{equation}\label{eq:harmonic_clausen_ke3}
E(k)K(k) =\frac{\pi^2(1-k^2)}{4}-\frac{\pi^2}{8}\sum_{m=1}^{\infty}\frac{(2m)!^3H_{m}W_{m}(k)}{2^{4m}m!^6}\left(k^2(1-k^2)\right)^{m},\quad 0<k<\frac{1}{\sqrt{2}} ,
\end{equation}
where $W_{m}(k)$ is defined by
\begin{align*}
W_{m}(k):=\frac{k^6(2m+3)(2m+1)^3-k^4(2m+1)^3(3m+5)}{(m+1)^2} \\
+\frac{k^2(12m^4+38m^3+38m^2+15m+2)-2m(m+1)^3}{(m+1)^2}.
\end{align*}

\end{corollary}
Furthermore, using \eqref{eq:jacobi_transke2} in \eqref{eq:harmonic_clausen_k3} and \eqref{eq:harmonic_clausen_ke3} yields
\begin{corollary}
\begin{equation}\label{eq:harmonic_clausen_k4}
K^2(k)=\frac{\pi^2}{4(1-k^2)}+\frac{\pi^2}{8(1-k^2)^3}\sum_{m=1}^{\infty}\frac{(-1)^m(2m)!^3H_{m}X_{m}(k)}{2^{4m}m!^6}\left(\frac{k}{k^2-1}\right)^{2m},
\end{equation}
for $0<k\le \sqrt{2}-1$, where the polynomial $X_{m}(k)$ is defined by
\begin{equation*}
X_{m}(k):=\frac{2m(m+1)^2(1+k^4)+k^2(4m^3+4m^2+2m+1)}{(m+1)^2}.
\end{equation*}
The product of the complete elliptic integrals of the first and second kind admits the expansion
\begin{equation}\label{eq:harmonic_clausen_ke4}
E(k)K(k) =\frac{\pi^2}{4(1-k^2)}+\frac{\pi^2}{8(1-k^2)^3}\sum_{m=1}^{\infty}\frac{(-1)^m(2m)!^3H_{m}Y_{m}(k)}{2^{4m}m!^6}\left(\frac{k}{k^2-1}\right)^{2m},
\end{equation}
for $0<k<\sqrt{2}-1$, where $Y_{m}(k)$ is defined by
\begin{align*}
Y_{m}(k):=\frac{2k^6m^2(m+1)^2+k^4(6m^4+18m^3+20m^2+9m+1)}{(m+1)^2} \\+
\frac{k^2(6m^4+20m^3+20m^2+9m+2)+2m(m+1)^3}{(m+1)^2}.
\end{align*}
\end{corollary}

\begin{corollary}
The following series for $1/\pi$ involving harmonic numbers hold true
\begin{equation}\label{eq:ramanujank3_h}
\sum_{m=1}^{\infty} \frac{(2m)!^3H_{m}}{2^{8m}m!^6}\left(\frac{144m^5-144m^4+40m^3-24m^2+m+1}{(2m-1)^3} \right)= \frac{4}{\pi},
\end{equation}
\begin{equation}\label{eq:berndtetal_h}
\sum_{m=1}^{\infty}\frac{(-1)^{m+1}(2m)!^3H_{m}}{2^{9m}m!^6}\left(\frac{432m^5-336m^4-40m^3+24m^2-m-1}{(2m-1)^{3}}\right)=\frac{2\sqrt{2}}{\pi},
\end{equation}
\begin{equation}\label{eq:ramanujank7_h}
\sum_{m=1}^{\infty} \frac{(2m)!^3H_{m}}{2^{12m}m!^6}\left(\frac{21168m^5-18816m^4+272m^3-180m^2+17m + 5}{(2m-1)^3}\right)= \frac{16}{\pi}.
\end{equation}
\end{corollary}
\begin{proof}
\noindent\textbf{Proof of \eqref{eq:ramanujank3_h}.}
Using Eqs.~\eqref{eq:rel_n3}, \eqref{eq:harmonic_clausen_k1} and \eqref{eq:harmonic_clausen_ke1} the result follows.

\medskip
\noindent\textbf{Proof of \eqref{eq:berndtetal_h}.}
The result follows from using Eqs.~\eqref{eq:rel_n4}, \eqref{eq:harmonic_clausen_k2} and \eqref{eq:harmonic_clausen_ke2}.

\medskip
\noindent\textbf{Proof of \eqref{eq:ramanujank7_h}.}
Using Eqs.~\eqref{eq:rel_n7}, \eqref{eq:harmonic_clausen_k1} and \eqref{eq:harmonic_clausen_ke1} the series is obtained.
\end{proof}

\begin{corollary}
The following series for $1/\pi$ involving harmonic numbers hold true
\begin{equation}\label{eq:ramanujank3_h1}
\sum_{m=1}^{\infty} \frac{(2m)!^3H_{m}}{2^{8m}m!^6}\left(\frac{144m^4+288m^3+136m^2-16m-7}{(m+1)^2} \right)= \frac{128}{\pi}-32,
\end{equation}
\begin{equation}\label{eq:berndtetal_h1}
\sum_{m=1}^{\infty}\frac{(-1)^{m+1}(2m)!^3H_{m}}{2^{9m}m!^6}\left(\frac{432m^4+960m^3+632m^2+112m+7}{(m+1)^{2}}\right)=64-\frac{128\sqrt{2}}{\pi},
\end{equation}
\begin{equation}\label{eq:ramanujank7_h1}
\sum_{m=1}^{\infty} \frac{(2m)!^3H_{m}}{2^{12m}m!^6}\left(\frac{21168m^4+44688m^3+25808m^2+2236m -47}{(m+1)^2}\right)= \frac{8192}{\pi}-2560.
\end{equation}
\end{corollary}
\begin{proof}
\noindent\textbf{Proof of \eqref{eq:ramanujank3_h1}.}
Using Eqs.~\eqref{eq:rel_n3}, \eqref{eq:harmonic_clausen_k3} and \eqref{eq:harmonic_clausen_ke3} the result follows.

\medskip
\noindent\textbf{Proof of \eqref{eq:berndtetal_h1}.}
The result follows from using Eqs.~\eqref{eq:rel_n4}, \eqref{eq:harmonic_clausen_k4} and \eqref{eq:harmonic_clausen_ke4}.

\medskip
\noindent\textbf{Proof of \eqref{eq:ramanujank7_h1}.}
Using Eqs.~\eqref{eq:rel_n7}, \eqref{eq:harmonic_clausen_k3} and \eqref{eq:harmonic_clausen_ke3} the series is obtained.
\end{proof}

\backmatter

\section*{Declarations}
\begin{itemize}
    \item \textbf{Acknowledgments:} Generative AI was used solely for English language translation and linguistic polishing of the manuscript. The author is grateful to Dr.~Ken Ono for his insightful comments on Lemma~\ref{lem:elliptic_relation_lemma} and for pointing out that the last six series presented in this work are new.
 \item \textbf{Funding:} The author received no financial support for the research, authorship, and/or publication of this article.
    \item \textbf{Conflict of interest:} The author declares no competing interests.
    \item \textbf{Ethical approval:} Not applicable.
    \item \textbf{Data availability:} Data sharing is not applicable to this article as no datasets were generated or analyzed during the current study.
    \item \textbf{Authors' contributions:} Entirely prepared by the sole author.
\end{itemize}


{\small
\begin{thebibliography}{99}

\bibitem{almkvistberndt1988}
Almkvist, G., Berndt, B.C.: Gauss, Landen, Ramanujan, the arithmetic-geometric mean, ellipses, $\pi$, and the Ladies' Diary. Amer. Math. Monthly \textbf{95}, 585--608 (1988)

\bibitem{baruahberndt2009eisenstein}
Baruah, N.D., Berndt, B.C.: Eisenstein series and Ramanujan-type series for $1/\pi$. The Ramanujan Journal \textbf{20}, 267--287 (2009)

\bibitem{bauer1859}
Bauer, G.: Von den Coefficienten der Reihen von Kugelfunctionen einer Variabeln. J. Reine Angew. Math. \textbf{56}, 101--121 (1859)

\bibitem{boyadzhiev2012}
Boyadzhiev, Khristo N.: Series with Central Binomial Coefficients, Catalan Numbers, and Harmonic Numbers. Journal of Integer Sequences, Vol. 15 (2012)

\bibitem{campbell2018}
Campbell, J.M.: Ramanujan-like series for $1/\pi$ involving harmonic numbers. The Ramanujan Journal \textbf{46}, 373--387 (2018)

\bibitem{clausen1828}
Clausen, T.: Ueber die Fälle, wenn die Reihe von der Form $y = 1 + \dots$ etc. ein Quadrat von der Form $z = 1 \dots$ etc. hat. J. Reine Angew. Math. \textbf{3}, 89--91 (1828)

\bibitem{gauss1813}
Gauss, C.F.: Disquisitiones generales circa seriem infinitam $1 + \frac{\alpha\beta}{1\cdot\gamma}x + \frac{\alpha(\alpha+1)\beta(\beta+1)}{1\cdot2\cdot\gamma(\gamma+1)}xx + \dots$ Comment. Soc. Regiae Sci. Gottingensis Rec. \textbf{2} (1813). Reprinted in: Werke, Vol. 3, pp. 123--162

\bibitem{guetzlaff1834}
Guetzlaff, C.: Aequatio modularis pro transformatione functionum ellipticarum septimi ordinis. J. Reine Angew. Math. \textbf{12}, 173--181 (1834)

\bibitem{houhewang2024}
Hou, Q.-H., He, H., Wang, X.: Ramanujan-inspired series for $1/\pi$ involving harmonic numbers. J. Difference Equ. Appl. \textbf{30}, 754--766 (2024)

\bibitem{jacobi1829}
Jacobi, C.G.J.: Fundamenta nova theoriae functionum ellipticarum. Fratrum Borntraeger, Regiomonti (1829)

\bibitem{landen1775}
Landen, J.: An investigation of a general theorem for finding the length of any arc of any conic hyperbola, by means of two elliptic arcs, with some other new and useful theorems deduced therefrom. Philos. Trans. R. Soc. Lond. \textbf{65}, 283--289 (1775)

\bibitem{legendre1825}
Legendre, A.M.: Traité des fonctions elliptiques et des intégrales eulériennes, Vol. 1. Imprimerie de Huzard-Courcier, Paris (1825)

\bibitem{lichu2024}
Li, C., Chu, W.: Series involving cubic central binomial coefficients of convergence rate $1/64$. Bull. Malays. Math. Sci. Soc. \textbf{47}, Article 94 (2024)

\bibitem{ramanujan1914}
Ramanujan, S.: Modular equations and approximations to $\pi$. Quart. J. Pure Appl. Math. \textbf{45}, 350--372 (1914)

\bibitem{sunzhou2026}
Sun, Z.-W., Zhou, Y.: Series involving central binomial coefficients and higher-order harmonic numbers. Preprint at arXiv:2602.12091 (2026)

\bibitem{weber1895}
Weber, H.: Lehrbuch der Algebra, Vol. 3. Vieweg und Sohn, Braunschweig (1895). Reprinted by Chelsea, New York (1961)
\end{thebibliography}
}
\end{document}